\documentclass [12pt,a4paper]{article}
\usepackage[T1]{fontenc}
\usepackage[ansinew]{inputenc}
\usepackage{amssymb}
\usepackage{amsmath}
\usepackage{graphicx}
\usepackage{amsthm}
\usepackage{cite}

\newtheorem{mytheorem}{\textbf{Theorem}}[section]
\newtheorem{mylemma}{\textbf{Lemma}}[section]

\newtheorem{remark}{\textbf{Remark}}[section]

\def\be{\begin{equation}}
\def\ee{\end{equation}}
\def\bea{\begin{eqnarray}}
\def\eea{\end{eqnarray}}
\def\bt{\begin{theorem}}
\def\et{\end{theorem}}
\def\bl{\begin{lemma}}
\def\el{\end{lemma}}
\def\br{\begin{remark}}
\def\er{\end{remark}}
\def\bc{\begin{corollary}}
\def\ec{\end{corollary}}
\def\bd{\begin{definition}}
\def\ed{\end{definition}}

\def\s0t{\sup _{0\leq \tau\leq t}}
\def\C0T{C([0,T];\,}

\def\DAS{D( A^{\frac{S}{2}})}
\def\DAS1{D( A^{ \frac{S+1}{2}})}

\def\dis{\displaystyle}

 \def\Uo1{\|U_0 \| _{W^{1,1}}}
 \def\Uo2{\|U_0\| _{W^{2,1}}}

 \def\e1t{e^{\lambda_1 t}}
 \def\e2t{e^{\lambda_2 t}}
 \def\e3t{e^{\lambda_3 t}}
 \def\l1a{\lambda _1}
 \def\l2a{\lambda _2}
 \def\l3a{\lambda _3}
\def\l123a{(\lambda _1-\lambda _2)(\lambda _1-\lambda _3)}
\def\l213a{(\lambda _2-\lambda _1)(\lambda _2-\lambda _3)}
\def\l312a{(\lambda _3-\lambda _1)(\lambda _3-\lambda _2)}
\def\la1x{(\lambda _1 + k\xi ^2)}
\def\la2x{(\lambda _2 + k\xi ^2)}
\def\la3x{(\lambda _3 + k\xi ^2)}
\def\kx2{k \xi ^2}
\def\non{\nonumber }

\begin{document}

\begin{center}
   {\large  CONVERGENCE TO EQUILIBRIUM FOR THE SEMILINEAR PARABOLIC EQUATION
WITH DYNAMICAL BOUNDARY CONDITION}\\[1cm]
                   \end{center}

                   \begin{center}
                  {\sc Hao Wu}\\
 School of Mathematical Sciences, Fudan University \\
Handan Road 220, 200433 Shanghai, P.R. China\\
 Email: haowufd@yahoo.com
                   \vspace{0.2cm}
\end{center}
       \vspace{1cm}


\begin{abstract}
\noindent This paper is concerned with the asymptotic behavior of
the solution to
 the following semilinear parabolic equation
 \be  u_t -\Delta u  + f(u) = 0, \qquad (x,t) \in \Omega
 \times \mathbb{R}^+ ,\non
 \ee
subject to the dynamical boundary condition \be
  \partial_\nu u+ \mu u + u_t = 0, \qquad (x,t) \in \Gamma
 \times \mathbb{R}^+,
  \non \ee and the initial condition \be u|_{t=0} =
u_0(x),\quad x \in \Omega, \non\ee where $\Omega\subset
\mathbb{R}^n$ $(n\in\mathbb{N}^+)$ is a bounded domain with smooth
boundary $\Gamma$, $\nu$ is the outward normal direction to the
boundary and $\mu\in \{0,1\}$. $f$ is analytic with respect to
unknown function $u$. Our main goal is to prove the convergence of a
global solution to an equilibrium as time goes to infinity by means
of a suitable \L ojasiewicz-Simon type inequality.

\end{abstract}


\section{Introduction}
\setcounter{equation}{0}
\par In this paper we consider the following semilinear parabolic equation
 \be  u_t -\Delta u  + f(u) = 0, \qquad (x,t) \in \Omega
 \times \mathbb{R}^+ ,\label{1.1}
 \ee
subject to the dynamical boundary condition \be
\partial_\nu u+ \mu u + u_t = 0, \qquad \qquad (x,t) \in \Gamma
 \times \mathbb{R}^+,
\label{1.2} \ee and the initial condition \be u|_{t=0} = u_0(x),
\qquad x \in \Omega. \label{1.3} \ee In the above, $\Omega\subset
\mathbb{R}^n$ $(n\in\mathbb{N}^+)$ is a bounded domain with smooth
boundary $\Gamma$. $\nu$ is the outward normal direction to the
boundary and $\mu\in \{ 0,1\}$.

Parabolic equation and systems with dynamical boundary conditions
have been extensively studied in the literature (see for instance,
\cite{AQB,BS,E2,E3,E4,FQ1,Hi,JG} etc.). In particular, local
existence and uniqueness of solution to general quasilinear
parabolic equation (systems) with
 dynamical boundary condition has been established in a series of papers by
 Escher \cite{E2,E3,E4} (see also \cite{Hi} for a semigroup approach in the $H^2_p(\Omega)$-setting, and \cite{BS}
 for the solvability
 result in a weighted H\"older space).
 Moreover, in \cite{E3,E4}, the author showed that the
 solution defines a local semiflow on certain Bessel potential
 spaces including the standard Sobolev space $W^{1,p}$. Based on the approach in \cite{E2,E3,E4} etc.,
 in \cite{FQ1} the authors studied the wellposedness of a semilinear
 parabolic equation subject to a nonlinear dynamical boundary condition with subcritical growth for nonlinearities
 and then investigated the large time behavior of the solution such as blow-up phenomenon.
 In our present paper, we are now interested in the question  whether the global
 solution to a class of semilinear parabolic equation with subcritical nonlinearity and dynamical
 boundary condition \eqref{1.1}--\eqref{1.3} will converge to an equilibrium as time goes to infinity.

Before stating our main results we make some assumptions
on nonlinearity $f$. \\
(\textbf{F1}) $f(s)$ is analytic for $s\in \mathbb{R}$. \\
 (\textbf{F2}) $$|f'(s)|\leq c(1+|s|^{p}),\;\;\;\; \forall
  s\in \mathbb{R},\ p\in [0,\alpha),$$
  where
  \begin{equation}\alpha:=
  \begin{cases}+\infty & n=1,2,\\ \frac{4}{n-2}& n\geq 3.
  \end{cases}\non
  \end{equation}
 (\textbf{F3})
 \begin{itemize}
 \item for $\mu=1$,\\
 \begin{equation}\liminf_{|s|\rightarrow
  +\infty}\frac{f(s)}{s}> -\frac{1}{4}\lambda, \nonumber
  \end{equation}
  where $\lambda>0$ is the best
  Sobolev
  constant in the following imbedding inequality \be \int_\Omega |\nabla
  u|^2 dx + \int_\Gamma u^2 dS \geq \lambda \int_\Omega u^2 dx;
  \nonumber \ee
\item for $\mu=0$,\\
\begin{equation}\liminf_{|s|\rightarrow
  +\infty}\frac{f(s)}{s}> 0. \nonumber
  \end{equation}
  \end{itemize}
\br Assumption (\textbf{F1}) is made so that we can  use  an
extended \L ojasiewicz-Simon inequality to prove our convergence
result (see Theorem \ref{con2} below). Assumption (\textbf{F2})
implies that the nonlinear term has a subcritical growth. Under
this assumption, we are able to prove the local existence and
uniqueness of solutions by adopting the argument in \cite{FQ1}
with some modifications. Moreover, due to the subcritical growth,
 existence of a strong solution to stationary problem can be
obtained by variational method. Assumption (\textbf{F3}) is
 a  condition needed to ensure  global existence of
solution to our problem \eqref{1.1}-\eqref{1.3}. For simplicity of
exposition the nonlinear term $f$ is assumed to depend only on
$u$. However, the results in this paper still remain true for the
nonlinear term $f(x,u)$ with additional smoothness assumption with
respect to $x$.\er Throughout this paper we simply denote the norm
in $L^2(\Omega)$ by
 $\parallel \cdot
 \parallel$ and  equip $H^1(\Omega)$ with the equivalent norm
$$\|u\|_{H^1(\Omega)}=\left(\int_\Omega|\nabla u|^2dx+ \int_\Gamma
 u^2dS\right)^{1/2}.$$

We are now in a position to state our main results.
 \begin{mytheorem}\label{glo}
  Suppose (\textbf{F1})(\textbf{F2})(\textbf{F3}) are satisfied. Then for
 any initial data $u_0\in H^1(\Omega)$, problem
 (\ref{1.1})--(\ref{1.3}) admits a unique global
  solution such that
$$ u\in C([0, +\infty);H^1(\Omega))\cap C((0, +\infty);H^2(\Omega))\cap
 C^1((0,+\infty);L^2(\Omega)),$$
$$\gamma(u)\in C([0,+\infty);H^{\frac{1}{2}}(\Gamma)) \cap C((0,+\infty);H^{\frac{3}{2}}(\Gamma))\cap
 C^1((0,+\infty);H^{\frac{1}{2}}(\Gamma)),$$
 where $\gamma(u)$ is the trace operator.
Moreover, if $u_0\in H^2(\Omega)$, then
$$ u\in C([0, +\infty);H^2(\Omega)),\ \ \gamma(u)\in C([0,+\infty);H^{\frac{3}{2}}(\Gamma)).$$
 \end{mytheorem}

 \begin{mytheorem}\label{con2}
  Suppose that (\textbf{F1}) (\textbf{F2}) (\textbf{F3}) are satisfied. Then for
 any initial data $u_0\in H^1(\Omega)$, and in addition $u_0\in L^\frac{4n}{n-2}(\Omega)$, $\gamma(u_0)\in
L^\frac{4n}{n-2}(\Gamma)$ when $n\geq 3$, the global
  solution to problem (\ref{1.1})--(\ref{1.3})
  converges to an equilibrium $\psi(x)$ in the topology of $H^1(\Omega)$ as time goes to infinity, i.e., \be
  \dis{\lim_{t\rightarrow +\infty}}(\|u(t, \cdot)-\psi\|_{H^1(\Omega)} + \|u_t(t, \cdot)\|)=0. \label{1.5a}\ee
  Here
  $\psi(x)$ is an equilibrium to problem
  (\ref{1.1})--(\ref{1.3}), i.e., $\psi(x)$ is a strong solution to the following
  nonlinear elliptic boundary value
  problem:
  \be \left\{\begin{array}{l}  - \Delta \psi + f(\psi)=0,\quad x\in \Omega, \\
   \partial_\nu \psi + \mu \psi=0, \;\;\; x\in \Gamma.\\
   \end{array}
   \label{1.6}
  \right.
 \ee
 \end{mytheorem}
Before giving the detailed  proof of our convergence theorem, we
briefly recall some related results in the literature.  For the
semilinear parabolic equation with homogeneous Dirichlet
 boundary condition, concerning
convergence to an equilibrium as time goes to infinity, we notice
that in one spacial dimension, the situation is much simpler: the
set of stationary states is discrete and one can take advantage of
 existence of a Lyapunov functional to prove that any bounded
global solution will converge to an equilibrium (see e.g.
\cite{M78,Ze}).  In the case of higher spatial dimension, it was
proved in \cite{S83,J981} by use of the so-called
{\L}ojasiewicz-Simon inequality that any bounded global
 solution will
 converge to a single stationary state provided
that $f$ is real analytic. This kind of result is highly
nontrivial, since in higher spatial dimension the stationary
states can form a continuum (see, e.g., \cite{Ha}). Moreover, a
counterexample for semilinear parabolic equations with $C^\infty$
nonlinearities are given in the literature (see \cite{PS},  also
\cite{PR96}) to show that there is
 a bounded global solution whose $\omega$-limit set is
diffeomorphic to the unit circle $S^1$.  After this breakthrough,
many further contributions were made  for various types of
nonlinear evolution equations (see, e.g.,
\cite{AFI,FIP1,FS,HM,Hu,J982,RH,WZ1,WZ2,WGZ,WGZ2} and references
therein).

Some new features of our present work are as follows. First, due to
the presence of dynamical term $u_t$, it turns out that for the
corresponding elliptic operator, the dissipative boundary condition
should be considered as  non-homogeneous. As a result, the \L
ojasiewicz-Simon type inequality we are going to derive is naturally
different from the one with the usual homogeneous boundary
conditions in the literature. Furthermore, we are able to treat both
non-homogenous Neumann and Robin boundary conditions (corresponding
to $\mu=0,1$, respectively). Second, the subcritical growth of
nonlinearity seems natural as stated in Remark 1.1. Besides, we do
not have any restriction on spatial dimension. Thus, our result
improves the growth condition of nonlinear term stated in
\cite[Chapter 3, Section 3.2]{Hu} even for the Dirichlet boundary
condition.

The remaining  part of this paper is organized as follows. In
Section 2 we prove the existence and uniqueness of the global
solution (Theorem \ref{glo}) and derive some uniform \textit{a
priori} estimates. In Section 3 we first study the stationary
problem. Then we proceed to establish the generalized \L
ojasiewicz--Simon type inequality, and to complete the proof of
Theorem \ref{con2}.


\section{Global Existence and Uniqueness}
\setcounter{equation}{0}

\par
Let \bea E&:=&\{(u,v)^T\in H^2(\Omega)\times
H^{\frac{3}{2}}(\Gamma):\
\gamma(u)=v\};\non\\
E_1&:=&\{(u,v)^T\in H^1(\Omega)\times H^{\frac{1}{2}}(\Gamma):\
\gamma(u)=v\};\non\\
F&:=&L^2(\Omega)\times H^{\frac{1}{2}}(\Gamma).\non\eea Following
\cite{E2,E3,E4,FQ1} we can rewrite problem \eqref{1.1}-\eqref{1.3}
in the following form:

\begin{equation}
    \left\{\begin{array}{l}  \frac{d}{dt}\left(
  \begin{array}{c}
    u\\v
 \end{array}\right)
 +A\left(
   \begin{array}{c}
    u\\v
   \end{array}\right)
  +\left(
   \begin{array}{c}
    f(u)\\0
   \end{array}\right)=0, \\
    (u(0),v(0))^T=(u_0,\gamma(u_0))^T,\\
    \end{array}
    \right.\label{2.1}
 \end{equation}
 with
\begin{eqnarray}
 A\left(
  \begin{array}{c}
    u\\v
 \end{array}\right)
 =\left(
   \begin{array}{c}
    -\Delta u\\ \mu v+ \gamma(\partial_\nu u)   \end{array}\right),\non
  \end{eqnarray}
and $E$ is the domain of $A$, i.e., $D(A)=E$.

By the results in Escher \cite{E4} and Fila $\&$ Quittner
\cite{FQ1}, we have
\begin{mytheorem}
Let $u_0\in H^1(\Omega)$ and assumptions
(\textbf{F1})(\textbf{F2}) be satisfied. Then problem
(\ref{1.1})--(\ref{1.3}) admits a unique maximal solution $u(t)$
such that
$$Z:=(u,\gamma(u))^T\in C([0,T_{max});E_1)\cap C((0,T_{max});E)
\cap C^1((0,T_{max});F).$$ Moreover, the solution $u(t)$ exists
globally if $u(t)$ is bounded in $H^1(\Omega)$. In addition, if
$u_0\in H^2(\Omega)$, then $Z\in C([0,T_{max});E)\cap
C^1([0,T_{max});F)$.
\end{mytheorem}
\begin{proof}
Let $u_0\in H^1(\Omega)$. By assumption (\textbf{F2}), the
condition $(W)$ in Escher \cite{E4} is satisfied. Thus $Z\in
C([0,T_{max});E_1)$ follows from \cite[Theorem 1]{E4}. On the
other hand, it follows from \cite[Theorem 2.2]{FQ1} that $Z\in
C((0,T_{max});E) \cap C^1((0,T_{max});F)$ and furthermore $Z\in
C([0,T_{max});E)\cap C^1([0,T_{max});F)$, if $u_0\in H^2(\Omega)$.
\end{proof}
\noindent Define
\begin{equation}
E(u(t))=\frac{1}{2}\int_\Omega|\nabla u|^2dx+
\frac{\mu}{2}\int_\Gamma u^2 dS+\int_\Omega F(u(t))dx
\end{equation}
where $F(z)=\int_0^z f(s)ds$. \\It's easy to see that $\forall \
t\in(0, T_{max})$,
\begin{equation}
\frac{d}{dt}E(u(t))+
\|u_t\|^2+\|u_t\|_{L^2(\Gamma)}^2=0,\label{2.3}
\end{equation}
which implies that $E(u)$ is decreasing respect to time.\\
 Next we
prove the global existence of the solution to problem
(\ref{1.1})--(\ref{1.3}).
\begin{mytheorem}\label{glo1}
Assume  (\textbf{F1})(\textbf{F2})(\textbf{F3}) hold. Then the
solution $u(t)$ obtained in Theorem 2.1 exists globally, i.e.,
$T_{max}=+\infty$. Furthermore, assuming in addition that $u_0\in
L^\frac{4n}{n-2}(\Omega)$, $\gamma(u_0)\in L^\frac{4n}{n-2}(\Gamma)$
when $n\geq 3$, then for any $\sigma>0$ we have the following
uniform estimate
$$\|u\|_{H^2(\Omega)}\leq C_\sigma,\qquad \forall \ t\geq \sigma,$$ where $C_\sigma$ is
a positive constant depending only on $\|u_0\|_{H^1(\Omega)}$,
$\|u_0\|_{L^{\frac{4n}{n-2}}(\Omega)}$,
$\|\gamma(u_0)\|_{L^{\frac{4n}{n-2}}(\Gamma)}$, $\sigma$ and
$|\Omega|$.
\end{mytheorem}
\begin{proof}
 By Theorem 2.1, in order to get the global
existence, it suffices to obtain uniform \textit{a priori}
estimates on $\|u\|_{H^1(\Omega)}$.\\
From the Sobolev imbedding theorem and the growth assumption
(\textbf{F2}), we can get
\begin{eqnarray}
\int_\Omega F(u)dx & =& \int_\Omega \int_0^u f(s)ds\nonumber\\
& \leq & C\int_\Omega \int_0^{|u|}(1+|s|^{p+1})dsdx\leq
C\int_\Omega(|u|+|u|^{p+2})dx\nonumber\\
& \leq & C(\|u\|_{L^1(\Omega)}+\|u\|^{p+2}_{L^{p+2}(\Omega)})\leq
C(\|u\|_{H^1(\Omega)}).\label{2.5}
\end{eqnarray}
This implies \be E(u(t))\leq C(\|u\|_{H^1(\Omega)}).\label{2.6}\ee
In the same way as in \cite{CEL}, we can also get an estimate in
the opposite direction. \\To see this, for sufficiently small
$\delta>0$, we consider two cases.

\noindent \textbf{Case 1}: $\mu=1$.\\ From the assumption
(\textbf{F3}), there exists $N=N(\delta)$ being a positive
    constant such that $f(z)/z \geq -\lambda +2\delta $ for $|z|\geq N$.
    Then we have \be F(s) =\int_0^N\frac{f(z)}{z} z
    dz + \int_N^s \frac{f(z)}{z} z dz \geq -\frac{\lambda-\delta}{2}s^2\ee
    for $|s|\geq \left(\frac{1}{\delta}(N^2-2C)\right)^{1/2}:=M$, where $C=\int_0^N\frac{f(z)}{z} z
    dz$.  \\For
    negative $s$ one can repeat the same computation with $N$ replaced
    by $-N$. Now we have \be \int_\Omega F(u) dx= \int_{|u|\leq M}
    F(u)dx + \int_{|u|> M}F(u)dx \geq -\frac{\lambda-\delta}{2}\int_\Omega u^2dx+
    C(|\Omega|,f)\ee
    where $C(|\Omega|,f)=|\Omega|\displaystyle{\min_{|s|\leq M}}F(s)$.\\
  Thus  we can deduce that
     \be E(u(t))\geq \varepsilon\left(\frac{1}{2}\int_\Omega |\nabla u|^2 dx
    + \frac{1}{2}\int_\Gamma u^2 dS\right) +C(|\Omega|,f)\label{bl}\ee
    provided $\varepsilon\leq \frac{\delta}{2\lambda}$.
    Taking $C=1/\varepsilon$, we get
\be \frac{1}{2}\|u\|^2_{H^1(\Omega)}=\frac{1}{2}\int_\Omega
|\nabla u|^2 dx
    + \frac{1}{2}\int_\Gamma u^2 dS\leq
    C(E(u(t))-C(|\Omega|,f)).\label{2.10}\ee
\noindent\textbf{Case 2}: $\mu=0$.\\
By assumption (\textbf{F3}), there exists $N=N(\delta)$ being a
positive
    constant such that $f(z)/z \geq 2\delta $ for $|z|\geq N$.
    Then we have \be F(s) =\int_0^N\frac{f(z)}{z} z
    dz + \int_N^s \frac{f(z)}{z} z dz \geq \frac{\delta}{2}s^2\ee
    for $s^2\geq \frac{1}{\delta}(N^2-2C)$.  Similar to Case 1, we have \be \int_\Omega F(u) dx= \int_{|u|\leq M}
    F(u)dx + \int_{|u|> M}F(u)dx \geq \frac{\delta}{2}\int_\Omega u^2dx+
    C(|\Omega|,f)\ee
    where $C(|\Omega|,f)=|\Omega|\displaystyle{\min_{|s|\leq M}}F(s)-\frac{\delta}{2}M^2|\Omega|$.\\
   Thus we can deduce that
     \be E(u(t))\geq \left(\frac{1}{2}\int_\Omega |\nabla u|^2 dx
    + \frac{\delta}{2}\int_\Omega u^2 dx\right) +C(|\Omega|,f)\label{bla}\ee
    Taking $C=1/\delta$, we get
\be \frac{1}{2}\|u\|^2_{H^1(\Omega)}=\frac{1}{2}\int_\Omega |\nabla
u|^2 dx
    + \frac{1}{2}\int_\Omega u^2 dx\leq
    C(E(u(t))-C(|\Omega|,f)).\label{2.10a}\ee
For both cases, we have  \bea \|u\|^2_{H^1(\Omega)} & \leq &
C(E(u(t))-C(|\Omega|,f))\leq
C(E(u_0)-C(|\Omega|,f))\nonumber\\
& \leq& C(\|u_0\|_{H^1(\Omega)})-C(|\Omega|,f) \non\eea This
implies the following uniform estimate: \be
\|u\|_{H^1(\Omega)}\leq M, \qquad \forall t\in(0,T_{max})\ee where
M does not depend on $T_{max}$. According to Theorem 2.1, this
leads to $T_{max}=+\infty$, i.e., the solution exists globally.

In order to get the uniform bound of $u(t)$ in $H^2(\Omega)$ norm,
we are going to get some higher-order estimates via a formal
argument. However, this procedure can be made rigorously within an
appropriate regularization scheme (e.g. \cite{Z1}, Chapter 6).

In what follows we only consider the case $n\geq 3$ since the
cases $n=1,2$ we can get the same result with a simpler proof (see
Remark 2.1).

Again we consider two cases.\\
\textbf{Case 1} $\mu=1$. Multiplying (\ref{1.1}) by $|u|^{p-2}u$,
integrating over $\Omega$, we get
\be\frac{1}{p}\frac{d}{dt}\left(\int_\Omega |u|^{p}dx+\int_\Gamma
|u|^{p}dS\right)+(p-1)\int_\Omega |u|^{p-2}|\nabla
u|^2dx+\int_\Omega
f(u)u|u|^{p-2}dx+\int_\Gamma |u|^{p}dS = 0\ee %
Taking $p=\frac{4n}{n-2}$, from assumption (\textbf{F3}), we can
deduce that for $\delta>0$ sufficiently small there exists
$N=N(\delta)>0$ such that
$$ \frac{f(u)}{u}\geq \left(-\frac{4(p-1)}{p^2}\lambda+\delta\lambda\right)
\qquad \text{for} \ |u|\geq N,\ n\geq 3,$$
and as a result \be f(u)u\geq
\left(-\frac{4(p-1)}{p^2}\lambda+\delta\lambda\right)u^2\qquad
\text{for} \ |u|\geq N.\label{eq}\ee \eqref{eq} implies that \be
\int_\Omega f(u)u|u|^{p-2}dx\geq C_f+
\left(-\frac{4(p-1)}{p^2}\lambda+\delta\lambda\right)\int_{|u|\geq
N} |u|^pdx\ee
    where $C_f=|\Omega|\displaystyle{\min_{|s|\leq
    N}}f(s)s|s|^{p-2}$.
Then we get \bea&&\frac{1}{p}\frac{d}{dt}\left(\int_\Omega |u|^p
dx+\int_\Gamma |u|^p dS\right)+\frac{4(p-1)}{p^2}\int_\Omega
|\nabla |u|^{\frac{p}{2}}|^2dx+\int_\Gamma |u|^p dS \non\\ &\leq&
\left(\frac{4(p-1)}{p^2}\lambda-\delta\lambda\right) \int_\Omega
|u|^p dx +C.\label{2.26}\eea Using assumption (\textbf{F3}) and
applying the imbedding inequality \be \int_\Omega |\nabla
  w|^2 dx + \int_\Gamma w^2 dS \geq \lambda \int_\Omega w^2 dx
  \nonumber \ee
 to $|u|^{p/2}$, we get \be \frac{d}{dt}\left(\int_\Omega
|u|^p dx+\int_\Gamma |u|^p dS\right)+C_1\left( \int_\Omega |u|^p
dx+\int_\Gamma
|u|^p dS\right) \leq C_2.\ee \\
Integrating with respect to $t$ yields that \be \int_\Omega
|u|^{p}dx+\int_\Gamma |u|^{p}dS\leq e^{-C_1t}\left(\int_\Omega
|u_0|^{p}dx+\int_\Gamma
|u_0|^{p}dS\right)+\frac{C_2}{C_1}.\label{ppp}\ee Thus, we get \be
\int_\Omega |u|^{\frac{4n}{n-2}}dx\leq C\label{2.30a}\ee which
implies \be\|f'(u)\|_{L^n(\Omega)}\leq C.\label{2.30}\ee Here $C$
is a constant depending on $\|u_0\|_{L^{\frac{4n}{n-2}}(\Omega)}$,
$\|\gamma(u_0)\|_{L^{\frac{4n}{n-2}}(\Gamma)}$ and $|\Omega|$.\\
Differentiating (\ref{1.1}) with respect to $t$, we get \be
u_{tt}-\Delta u_{t}+ f'(u)u_t=0.\label{d1}\ee Multiplying
\eqref{d1} by $u_t$ and integrating over $\Omega$, we obtain
\be\frac{1}{2}\frac{d}{dt}(\parallel u_t\parallel^2+\parallel
u_t\parallel^2_{L^2(\Gamma)})+\parallel \nabla u_t\parallel^2+
\|u_t\|_{L^2(\Gamma)}^2+\int_\Omega f'(u) u_t^2dx=0.\label{ut1}\ee
It follows from H\"{o}lder's inequality and the Gagliado-Nirenberg
inequality that
\bea &&\left|\int_\Omega f'(u)u_t^2dx\right|\leq \|f'(u)\|_{L^n}\|u_t\|^2_{L^\frac{2n}{n-1}}\non\\
&\leq& C  \|f'(u)\|_{L^n}( \|\nabla u_t\|\|u_t\|+\|u_t\|^2)\non\\
&\leq& \frac{1}{2} \|\nabla u_t\|^2 +C\|u_t\|^2 .\label{ettt}\eea
Hence, we have
\begin{eqnarray}
& &t\left(\parallel u_t\parallel^2+\parallel
u_t\parallel^2_{L^2(\Gamma)}\right)+\int_0^t\tau \left(\parallel
\nabla
u_t\parallel^2+\parallel u_t\parallel^2_{L^2(\Gamma)}\right)d\tau\nonumber\\
&\leq& C\int_0^t\tau \|u_t\|^2d\tau+
\int_0^t\left(\|u_t\|^2+\parallel
u_t\parallel^2_{L^2(\Gamma)}\right)d\tau.\label{aa}
\end{eqnarray}
From \eqref{2.3} and \eqref{2.6} we can conclude that for any $t>0$
\be t\left(\parallel u_t\parallel^2+\parallel
u_t\parallel^2_{L^2(\Gamma)}\right)\leq C(1+t),\ \
\int_0^t\tau\left(\parallel \nabla u_t\parallel^2+\parallel
u_t\parallel^2_{L^2(\Gamma)}\right)d\tau\leq C(1+t).\label{2.18}\ee
Multiplying (\ref{d1}) by $-\Delta u_t$ and integrating over
$\Omega$, then from the boundary condition and H\"{o}lder's
inequality we get
\begin{eqnarray}
 & &\frac{1}{2}\frac{d}{dt}\left(\parallel
\nabla u_t\parallel^2+\parallel
u_t\parallel^2_{L^2(\Gamma)}\right)+\parallel
u_{tt}\parallel^2_{L^2(\Gamma)}+
\parallel\Delta u_t\parallel^2\nonumber\\
& \leq & \parallel f'(u)\parallel_{L^n(\Omega)}\ \parallel
u_t\parallel_{L^{\frac{2n}{n-2}}(\Omega)}\parallel \Delta
u_t\parallel.\label{2.21}
\end{eqnarray}
By (\ref{2.30}) and the Sobolev imbedding theorem, (\ref{2.21})
yields that\be \frac{1}{2}\frac{d}{dt}\left(\| \nabla u_t\|^2+ \|
u_t\|^2_{L^2(\Gamma)}\right)+ \|\Delta
u_t\|^2+\|u_{tt}\|^2_{L^2(\Gamma)}
 \leq \frac{1}{2}\parallel \Delta u_t\parallel^2+C\parallel
u_t\parallel_{H^1(\Omega)}^2\label{2.31} \ee Multiplying
(\ref{2.31}) by $t^2$ and integrating from 0 to $t$,
\begin{eqnarray}
& &t^2\left(\parallel\nabla u_t\parallel^2 + \parallel
u_t\parallel^2_{L^2(\Gamma)}\right) + \int_0^t \tau^2\parallel
\Delta u_t\parallel^2d\tau
+2\int_0^t \tau^2\|u_{tt}\|^2_{L^2(\Gamma)}d\tau\non\\
&\leq&\int_0^t \tau \parallel \nabla u_t
\parallel^2 d\tau + \int_0^t \tau\parallel
u_t\parallel^2_{L^2(\Gamma)} d\tau + 2C \int_0^t \tau^2
\parallel u_t\parallel_{H^1(\Omega)}^2 d\tau.\label{2.32}
\end{eqnarray}
Combining it with  (\ref{2.18}) yields that  \be t^2\parallel
u_t\parallel^2_{H^1(\Omega)}\ \leq\ C(1+t^2),\qquad \forall \
t>0.\label{2.34}\ee \textbf{Case 2} $\mu=0$.\\ Similarly we have
\be\frac{1}{p}\frac{d}{dt}\left(\int_\Omega |u|^{p}dx+\int_\Gamma
|u|^{p}dS\right)+(p-1)\int_\Omega |u|^{p-2}|\nabla
u|^2dx+\int_\Omega f(u)u|u|^{p-2}dx = 0,\ee where
$p=\frac{4n}{n-2}$. From (\textbf{F3}), for $\delta>0$ small there
exists $N=N(\delta)>0$ such that
 \be f(u)u\geq \delta u^2\qquad \text{for} \ |u|\geq
N,\ee which implies \be \int_\Omega f(u)u|u|^{p-2}dx\geq C_f+\delta
\int_{|u|\geq N} |u|^pdx\geq (C_f-\delta N^p|\Omega|)+\delta
\int_\Omega |u|^pdx,\ee
    where $C_f=|\Omega|\displaystyle{\min_{|s|\leq
    N}}f(s)s|s|^{p-2}$.
Then we get \be\frac{d}{dt}\left(\int_\Omega |u|^p dx+\int_\Gamma
|u|^p dS\right)+C\int_\Omega \left(|\nabla |u|^{\frac{p}{2}}|^2
+|u|^p \right)dx \leq C.\label{2.26a}\ee This immediately gives
\eqref{2.30}. \\
Furthermore we have the following corresponding inequalities
 to \eqref{aa} \eqref{2.32}:
\begin{eqnarray}
& &t\left(\parallel u_t\parallel^2+\parallel
u_t\parallel^2_{L^2(\Gamma)}\right)+\int_0^t\tau \parallel \nabla
u_t\parallel^2d\tau\nonumber\\
&\leq& C\int_0^t\tau \|u_t\|^2d\tau+
\int_0^t\left(\|u_t\|^2+\parallel
u_t\parallel^2_{L^2(\Gamma)}\right)d\tau.\label{aaa}
\end{eqnarray}
\begin{eqnarray}
& &t^2\parallel\nabla u_t\parallel^2  + \int_0^t \tau^2\parallel
\Delta u_t\parallel^2d\tau
+2\int_0^t \tau^2\|u_{tt}\|^2_{L^2(\Gamma)}d\tau\non\\
&\leq&2\int_0^t \tau \parallel \nabla u_t
\parallel^2 d\tau + 2C \int_0^t \tau^2
\parallel u_t\parallel_{H^1(\Omega)}^2 d\tau.\label{2.32a}
\end{eqnarray}
By using the equivalent $H^1(\Omega)$ norm of $u_t$, it is easy to
see that \eqref{2.34} holds as well. \\Now for both cases, we
consider the following elliptic boundary value problem
\begin{equation}\left\{
\begin{array}{c}
\Delta u=u_t+f(u),\\
\partial_\nu u+\mu u+u_t\mid_\Gamma=0,
\end{array}
\right.\label{ell}\end{equation} according to the elliptic
regularity theory and the trace theorem we have \be \|
u\|_{H^2(\Omega)}\leq
C(\|u_t+f(u)\|+\|u_t\|_{H^\frac{1}{2}(\Gamma)}+\|u\|) \leq
C(\|u_t\|_{H^1(\Omega)}+\|f(u)\|+\|u\|), \label{2.22}\ee where $C$
is
certain positive constant depending only on $\Omega$.\\
Combining the fact
$$\frac{2(n+2)}{n-2}< \frac{4n}{n-2}\qquad  \text{for}\  \ n\geq 3  $$
with the result (\ref{2.30a}), assumption (\textbf{F2}) and
Young's inequality yields that \be\parallel f(u)\parallel\ \leq\
C.\label{2.27}\ee
 By (\ref{2.34}), (\ref{2.22}), (\ref{2.27}), for any $\sigma>0$, we obtain \be\| u\|_{H^2(\Omega)}\ \leq C_\sigma,\qquad \forall \ t \geq
\sigma\label{bh2}\ee where $C_\sigma$
 is a constant depending on $\|u_0\|_{H^1(\Omega)}$,
$\|u_0\|_{L^{\frac{4n}{n-2}}(\Omega)}$,
$\|\gamma(u_0)\|_{L^{\frac{4n}{n-2}}(\Gamma)}$, $\sigma$ and
$|\Omega|$.

\end{proof}
\br  We consider the case $n\geq 3$ in the proof when we want to
get the estimate of $\|f(u)\|$ and also $\|f'(u)\|_{L^n(\Omega)}$.
For $n=1,2$, the proof is simpler because we can simply use
$\|u_0\|_{H^1(\Omega)}$ to get the desired estimates due to the
Sobolev imbedding inequality for $n=1,2$.\er %

\section{Proof of Theorem \ref{con2} }
\setcounter{equation}{0}
The proof of Theorem \ref{con2} consists of several steps.\\
\noindent \textbf{Step 1.}
\par From \eqref{2.3} it is clear that the energy $E(u(t))$ is decreasing with respect to
time. On the other hand, \eqref{bl} or \eqref{bla} implies that
$E(u(t))$ is bounded from below. Thus we know that $E(u(t))$
serves as a Lyapunov functional. Besides,
  $\forall\ t>0$ , if \ $E(u)\equiv E(u_0)$ then
$\frac{dE}{dt}\equiv 0$. This enables us to deduce that $u_t\equiv
0$, which means $u$ is an equilibrium. The $\omega$-limit set of
$u_0 \in H^1(\Omega)$ is defined as follows:
$$\omega(u_0)\!=\! \{\psi(x)\in\! H^1(\Omega):\ \exists \  \{t_n\}_{n=1}^\infty,\ t_n\rightarrow +\infty  \ s.\ t.\
u(x,t_n)\rightarrow \psi(x)\ \text{in}\ H^1(\Omega)\}.$$ Then from
the well-known results in the dynamic system (e.g.,\cite{TE}, Lemma
I.1.1) it is easy to see that \begin{mylemma} The $\omega$-limit set
of $u_0$ is a non-empty compact connected subset in $H^1(\Omega)$.
Furthermore, (i) it is invariant under the nonlinear semigroup
$S(t)$ defined by the solution $u(x,t)$, i.e, $S(t)\omega(u_0) =
\omega(u_0)$ for all $ t \geq 0$.  (ii) $E(u)$ is constant on
$\omega(u_0)$. Moreover, $\omega(u_0)$ consists of
equilibria.\end{mylemma}

\noindent \textbf{Step 2.} In this step, we collect some results on
the stationary problem.
\begin{mylemma}
Suppose that $\psi \in H^2(\Omega)$ is a strong solution to
problem (\ref{1.6}). Then $\psi$ is a critical point of the
functional $E(u)$ in $H^1(\Omega)$. Conversely, if $\psi$ is a
critical point of the functional $E(u)$ in $H^1(\Omega)$, then
$\psi \in H^2(\Omega)$, and it is a strong solution to problem
(\ref{1.6}).
\end{mylemma}
\begin{proof}
 If $\psi \in H^2(\Omega)$ satisfies
(\ref{1.6}), then for any $v \in H^1(\Omega)$, it follows from
(\ref{1.6}) that \be \int_{\Omega}(- \Delta \psi + f(\psi))vdx=0.
\non\ee By integration by parts and the boundary condition in
(\ref{1.6}), we get \be \int_{\Omega}(\nabla \psi \cdot \nabla v +
f(\psi) v)dx
    +\mu\int_{\Gamma}\psi v dS=0 \label{3.46}\ee
    which, by a straightforward calculation, is just
    the following
    \be
   \left. \frac{d E(\psi+\varepsilon v)}{d
    \varepsilon}\right|_{\varepsilon=0}=0.\non\ee
    Thus, $\psi$ is a critical point of $E(u)$. Conversely, if
    $\psi$ is a critical point of $E(u)$ in $H^1(\Omega)$, then
    (\ref{3.46}) is satisfied. By the assumption (\textbf{F2}) on subcritical growth
     and the bootstrap argument, $\psi\in H^2(\Omega)$ follows.
\end{proof}
    The following lemma claims that problem (\ref{1.6}) admits at
    least a strong solution.
\begin{mylemma}
The functional $E(u)$ has at least a
    minimizer $v\in H^1(\Omega)$ such that
    \be
    E(v)= \dis{\inf_{u\in H^1(\Omega)}}E(u).\non\ee
    In other words, problem (\ref{1.6}) admits at least a strong solution.
    \end{mylemma}
\begin{proof}

From Section 2 we can see that $E(u)$ is bounded from below on
$H^1(\Omega)$. Therefore, there is a minimizing sequence $u_n\in
H^1(\Omega)$ such that
     \be E(u_n) \rightarrow \dis{\inf_{u\in H^1}}E(u).\label{3.3}\ee
On the other hand, $E(u)$ can be written in the form:
    \be
    E(u)=\frac{1}{2}\|\nabla u\|^2+\frac{\mu}{2}\|u\|_{L^2(\Gamma)}^2+\mathcal{F}(u)\ee
    with
    \be
    \mathcal{F}(u)=\int_{\Omega}F(u)dx,\ee
       it follows from \eqref{bl} (\eqref{bla} for $\mu=0$)
       that $u_n$ is bounded in $H^1(\Omega)$. It turns out from the weak compactness
       that there is a subsequence, still denoted by $u_n$,
       such that $u_n$ weakly converges to $v$ in $H^1(\Omega)$. Thus,
       $v \in H^1(\Omega)$. We infer from the Sobolev imbedding theorem that the
       imbedding $H^1(\Omega)\hookrightarrow  L^\gamma(\Omega) $ ( $1\leq \gamma< \frac{n+2}{n-2}$)
        is compact. As a result, $u_n$ strongly converges to $v$ in $L^\gamma(\Omega)$. It turns out
        from the assumption (\textbf{F2}) that $\mathcal{F}(u_n) \rightarrow
    \mathcal{F}(v)$.
    Since $\|u\|^2_{H^1}$ is weakly lower semi-continuous, it follows
     from (\ref{3.3}) that $E(v)=\dis{\inf_{u\in H^1}}E(u)$.
     From the assumption (\textbf{F2}) on subcritical growth, by
     the
     bootstrap argument and elliptic regularity theorem, this weak solution is also a strong solution in
     $H^2(\Omega)$. Furthermore, we can also use the bootstrap argument to show that $v\in
     L^\infty(\Omega)$.
The proof is completed.
\end{proof}
\br Under assumption (\textbf{F2}), by a further bootstrap argument,
we can show that $\psi$ is also a classical solution.  \er

\noindent \textbf{Step 3.} Define \be  \mathcal{D}:=\{u \in
H^2(\Omega) \ |\
\partial_\nu u + \mu u \mid_\Gamma = 0 \}.  \non \ee

Let $\psi$ be a fixed critical point of
   $E(u)$. In what follows we are going to establish  the generalized \L ojasiewicz--Simon
   inequality which extends the original one by Simom \cite{S83} for the second
   order nonlinear  parabolic equation subject to the Dirichlet boundary condition. Similar inequalities have been derived in
 our  recent papers \cite{WZ1, WZ2, WGZ2}
   to prove the convergence to equilibrium for various evolution equations
   (systems) subject to the dynamical boundary conditions.
\begin{mylemma} \label{ls1}Let $\psi$ be a
   critical point of $E(u)$. Then there exist constants
    $\theta\in
   (0,\frac{1}{2})$ and $\beta>0$ depending on  $\psi$ such that for
   any $u \in
   H^2(\Omega)$ satisfying $\|
   u-\psi\|_{H^1(\Omega)}< \beta $ and  $\|
    \partial_\nu u +\mu u \|_{L^2(\Gamma)}<\beta$, we have
   \be
   \| -\Delta u+f(u) \|_{(H^1(\Omega))'} + \|
    \partial_\nu u +\mu u \|_{L^2(\Gamma)} \ \geq\  | E(u) - E(\psi)
   |^{1-\theta}.\label{SL1a}\ee
   \end{mylemma}
\begin{proof}
Let  \be M(u)=
  -\Delta u + f(u). \ee  Then $M(u)$ maps $u\in H^2(\Omega)$ into
  $L^2(\Omega)$. \\
First, we claim that there exist constants
    $\theta'\in
   (0,\frac{1}{2})$ and $\beta _1>0$ depending on $\psi$, $\forall w \in
   \mathcal{D}$ (i.e., $w$ satisfies the homogeneous Neumann or Robin boundary condition) satisfying $\|
   w-\psi \|_{H^1}< \beta _1 $, we have
   \be
   \| M(w)\|_{(H^1(\Omega))'} \ \geq\  C| E(w) - E(\psi)
   |^{1-\theta'}.\label{4.5w}\ee
The above claim follows from the same argument as that in
\cite{HM, J981} for the case of homogeneous Dirichlet boundary
condition (see also \cite{WGZ}), so the details can be omitted
here. However, we are now dealing with the nonhomogeneous boundary
condition. As a result, we have to extend the above result from
functions in $\mathcal{D}$ to $H^2(\Omega)$. Hereafter we always
keep in mind that $\psi$ is an equilibrium satisfying
$\partial_\nu \psi+\mu\psi=0$ on $\Gamma$.
As before we consider two cases. \\
\textbf{Case 1} $\mu=1$. For any $u\in H^2(\Omega)$, we consider the
following Robin Problem:
\be \left\{\begin{array}{l}  \Delta w =\Delta u,\quad x\in \Omega, \\
   \partial_\nu w + w=0, \;\;\; x\in \Gamma.\\
   \end{array}
   \nonumber
  \right.
 \ee
It has a unique solution $w$ belonging to $\mathcal{D}$. On the
other hand, we introduce the Robin map $R:\ H^s(\Gamma)\rightarrow
H^{s+(3/2)}(\Omega)$ which is defined as follows: \be Rp=q
\Longleftrightarrow
\left\{\begin{array}{l}  \Delta q=0 \ \ \text{in} \ \ \Omega, \\
   \partial_\nu q + q = p\ \ \text{on} \ \ \Gamma.\\
   \end{array}
   \nonumber
  \right.
\ee As shown in \cite{CEL}, $R$ is continuous for $s\in \mathbb{R}$.
Thus we have
 \be \|w-u\|_{H^1(\Omega)}\leq C \|\partial_\nu u+ u \|_{L^2(\Gamma)}, \label{vw}\ee
where the constant $C$ does not depend on $u$. As a result,  \be
\|w-\psi\|_{H^1(\Omega)}\leq
\|w-u\|_{H^1(\Omega)}+\|u-\psi\|_{H^1(\Omega)}\leq C \|\partial_\nu
u+ u \|_{L^2(\Gamma)}+\|u-\psi\|_{H^1(\Omega)}.\ee Then it is easy
to see the fact that $u\in H^2(\Omega)$ being in the small
neighbourhood of $\psi$ in $H^1(\Omega)$ and $\|\partial_\nu u+
u\|_{L^2(\Gamma)}$ being small imply that $w$ stays in a
small neighbourhood of $\psi$ in  $H^1(\Omega)$.\\
  By direct
calculation and the Sobolev imbedding theorem we can see that\bea
\| M(w) \|_{(H^1(\Omega))'}&\leq&
 \| M(u) \|_{(H^1(\Omega))'}+\|f(u)-f(w)\| _{(H^1(\Omega))'}\non \\
 &\leq &  \| M(u) \|_{(H^1(\Omega))'} + \left\|\int_0^1
 f'(u+t(w-u))(u-w)dt\right\|_{(H^1(\Omega))'}\non\\
 &\leq &  \| M(u) \|_{(H^1(\Omega))'} + \max_{0\leq t\leq1}\|f'(u+t(w-u))\|_{L^\frac{n}{2}(\Omega)}\|u-w\|_{L^\frac{2n}{n-2}(\Omega)}\non\\
  &\leq &  \| M(u) \|_{(H^1(\Omega))'} + C\|u-w\|_{H^1(\Omega)}\non\\
 &\leq& \| M(u) \|_{(H^1(\Omega))'}+C \|\partial_\nu u+ u \|_{L^2(\Gamma)}. \label{vw1}\eea
On the other hand, by the Newton-Leibniz formula and (\ref{vw}) we
can get \bea&&
 | E(w)- E(u)| \non\\
 & \leq & \left|\int_0^1\int_\Omega  M( u + t(w - u)) (u-w) dxdt\right|
 +\left|\int_0^1\int_\Gamma (1-t)(\partial_\nu u + u)(u-w)  dSdt\right|
 \non \\
& \leq & C \left(\| M(u)\|_{(H^1(\Omega))'}+\|\partial_\nu u+ u \|_{L^2(\Gamma)}\right)\|\partial_\nu u+ u \|_{L^2(\Gamma)}\non \\
& \leq & C \left(\| M(u)\|_{(H^1(\Omega))'}+\|\partial_\nu u+ u
\|_{L^2(\Gamma)}\right)^2. \label{vw2}\eea Since \be \mid E(w)-
E(\psi)\mid^{1-\theta'}\geq  \mid E(u)- E(\psi)\mid^{1-\theta'}-
\mid E(w)- E(u)\mid^{1-\theta'} \label{vw3},\ee and $0 < \theta'<
\frac{1}{2}$,  $ 2(1-\theta')-1>0$, then from  (\ref{4.5w}) and
(\ref{vw1})--(\ref{vw2}) we can conclude that \be C(\| M(u)
\|_{(H^1(\Omega))'}+\|
 \partial_\nu u + u
\|_{L^2(\Gamma)})\ \geq
 | E(u) - E(\psi) |
^{1-\theta'}, \non\ee for $ \|u-\psi\|_{H^1(\Omega)}< \beta$,
$\|\partial_\nu u+u\|_{L^2(\Gamma)}<\beta$,
where $\beta>0$ is chosen small enough such that $\|w-\psi\|_{H^1}< \beta_1$.\\
Next we choose $\varepsilon$, $0< \varepsilon<\theta'$ and $\beta$
smaller if necessary such that when $\|u-\psi\|_{H^1}< \beta$, \be
\ \frac{1}{C} \mid E(u) - E(\psi)|^{-\varepsilon}\geq 1.
\label{4.ddd}\ee Setting $\theta=\theta'
-\varepsilon\in(0,\frac{1}{2})$,  when $\|u-\psi\|_{H^1}< \beta$
and $\|
 \partial_\nu u + u \|_{L^2(\Gamma)}<\beta$, we finally have \be
\| M(u) \|_{(H^1(\Omega))'}+\|
 \partial_\nu u + u
\|_{L^2(\Gamma)} \ \geq
 |E(u) - E(\psi) |
^{1-\theta}. \label{4.30}\ee
 which is exactly (\ref{SL1a}).\\
\textbf{Case 2} $\mu=0$. For any $u\in H^2(\Omega)$, we consider the
following Neumann Problem:
\be \left\{\begin{array}{l}  -\Delta w+w = -\Delta u+u,\quad x\in \Omega, \\
   \partial_\nu w =0, \;\;\; x\in \Gamma.
   \end{array}
   \nonumber
  \right.
 \ee
It's easy to see that there exists a unique solution $w$ belonging
to $\mathcal{D}$ and \be \|w\|_{H^1(\Omega)}\leq C
\left(\|u\|_{H^1(\Omega)}+\|\partial_\nu u\|_{L^2(\Gamma)}\right).
\non\ee Again we can see that $u$ being in the small neighbourhood
of $\psi$ in $H^1(\Omega)$ with $\|\partial_\nu u\|_{L^2(\Gamma)}$
small implies that
$w$ stays in a small neighbourhood of $\psi$ in $H^1(\Omega)$.\\
Furthermore, by the energy estimate, we have \be
\|w-u\|_{H^1(\Omega)}\leq C \|\partial_\nu u \|_{L^2(\Gamma)}.
\label{vwa}\ee In a way similar to Case 1, the following estimates
follow: \bea &&\|  M(w) \|_{(H^1(\Omega))'}\non\\ &\leq&
 \| M(u) \|_{(H^1(\Omega))'}+\|(f(u)-u)-(f(w)-w)\| _{(H^1(\Omega))'}\non \\
 &\leq &  \| M(u) \|_{(H^1(\Omega))'} + \max_{0\leq t\leq 1}\|f'(u+t(w-u))\|_{L^\frac{n}{2}(\Omega)}\|u-w\|_{L^\frac{2n}{n-2}(\Omega)}+\|u-w\|_{(H^1(\Omega))'}\non\\
  &\leq &  \| M(u) \|_{(H^1(\Omega))'} + C\|u-w\|_{H^1(\Omega)}\non\\
 &\leq& \| M(u) \|_{(H^1(\Omega))'}+C \|\partial_\nu u \|_{L^2(\Gamma)}. \label{vw1a}\eea
On the other hand, by the Newton-Leibniz formula and (\ref{vwa}),
we can get \bea&&
 |E(w)- E(u)| \non\\
 & \leq & \left|\int_0^1\int_\Omega  M( u + t(w - u)) (u-w) dxdt\right|
 +\left|\int_0^1\int_\Gamma (1-t)\partial_\nu u(u-w)  dSdt\right|
 \non \\
& \leq & C \left(\| M(u)\|_{(H^1(\Omega))'}+\|\partial_\nu u
\|_{L^2(\Gamma)}\right)^2. \label{vw2a}\eea Then arguing as in Case
1, we can get the required corresponding result for $\mu=0$.

 Thus, the lemma is proved.
 \end{proof}
Following the idea in \cite{WZ2} (see also \cite{WGZ2}) we can
easily extend the previous lemma in such a way that we only need the
smallness of $\|u-\psi\|_{H^1(\Omega)}$,

\begin{mylemma}[Generalized \L ojasiewicz--Simon Inequality]  \label{ls}Let $\psi$ be a
   critical point of $E(u)$. Then there exist constants
    $\theta\in
   (0,\frac{1}{2})$ and $\beta_0>0$ depending on  $\psi$ such that for
   $\forall u \in
   H^2(\Omega)$, $\|
   u-\psi\|_{H^1(\Omega)}< \beta _0$, we have
   \be
   \| -\Delta u+f(u) \|_{(H^1(\Omega))'} + \|
    \partial_\nu u +\mu u \|_{L^2(\Gamma)} \ \geq\ | E(u) - E(\psi)
   |^{1-\theta}.\label{SL1}\ee
   \end{mylemma}

\begin{proof} Let $\beta>0$ and $\theta\in\left(0,\frac12\right)$ be the constants appearing in Lemma \ref{ls1}.
\\
We consider the following two cases: \\
(i) $\|u-\psi\|_{H^1(\Omega)}<\beta$ and $\|
    \partial_\nu u +\mu u \|_{L^2(\Gamma)}<\beta$. Then from Lemma \ref{ls1} we are
done.\\
(ii) $\|u-\psi\|_{H^1(\Omega)}<\beta$ but $\|
    \partial_\nu u +\mu u \|_{L^2(\Gamma)}\geq \beta$. By Sobolev imbedding
     $H^1(\Omega)\hookrightarrow L^\frac{2n}{n-2}(\Omega)$, there
exist $\widetilde{\beta}>0$ depending  on $\psi$ such that for any
$u$ satisfying $\|u-\psi\|_{H^1(\Omega)}<\widetilde{\beta}$,\be
|E(u)-E(\psi)|^{1-\theta}<\beta.\label{552}\ee Thus, we have
\begin{eqnarray}
\| -\Delta u+f(u) \|_{(H^1(\Omega))'} +
\|\partial_\nu u +\mu u \|_{L^2(\Gamma)} &\geq &\|\partial_\nu u + \mu u\|_{L^2(\Gamma)}\nonumber\\
& \geq & \beta\non\\
&> &|E(u)-E(\psi)|^{1-\theta}.
\end{eqnarray}
Finally by setting $\beta_0=\min\{\beta,\widetilde{\beta}\}$, the
lemma is proved.
\end{proof}

\noindent \textbf{Step 4.} After the previous preparations, we now
proceed to finish the proof of Theorem \ref{con2}, following a
simplified  argument introduced in \cite{J981} in which the key
observation is that after certain time $t_0$, the solution $u$
will fall into the small neighborhood of a certain equilibrium
$\psi(x)$, and stay there forever. \\First, it is easy to see that
$$\|u_t\|\rightarrow 0\ \ \ \text{as} \ t\rightarrow+\infty.$$
   From Step 1, there is a sequence $\{t_n\}_{n\in\mathbb{N}}, \, t_n\rightarrow
   +\infty$ such that
   \be u(x, t_n) \rightarrow \psi(x),\qquad \text{in}\ H^1(\Omega), \label{3.ee}\ee
   where $\psi(x)\in \omega(u_0)$ is a certain equilibrium.
   On the other hand, it follows from (\ref{2.3}) that $E(u)$ is
   decreasing in time and \be \lim_{n\rightarrow +\infty}E(t_n)=
E(\psi).\label{4.KKK}\ee We now consider all possibilities.

   (1). If there is a $t_0>0$ such that at this time
   $E(u)=E(\psi)$, then for all $t>t_0$, we deduce from
   (\ref{2.3}) that $u$ is independent of $t$. Since $u(x, t_n)
   \rightarrow \psi$, the theorem is proved.

   (2). If there is $t_0>0$ such that for all $t\geq t_0$,
   $u$ satisfies the condition of Lemma \ref{ls}, i.e.,
   $\|u-\psi\|_{H^1(\Omega)} < \beta_0$, then for $\theta \in (0,
   \frac{1}{2})$ introduced in Lemma  \ref{ls}, we have
   \be
   \frac{d}{dt}(E(u)-E(\psi))^{\theta}=\theta
   (E(u)-E(\psi))^{\theta-1}\frac{d E(u)}{dt}.\label{3.51}\ee
   Combining it with (\ref{SL1}), (\ref{2.3}) yields
\be \frac{d}{dt}(E(u) - E(\psi))^\theta +
\frac{\theta}{2}\left(\parallel -\Delta u+f(u) \parallel + \parallel
u_t
\parallel_{L^2(\Gamma)}\right)\ \leq\
0. \label{3.52}\ee Integrating from $t_0$ to $t$,   we get \be(E(u)
- E(\psi))^{\theta} + \frac{\theta}{2} \int_{t_0}^t \left(\parallel
-\Delta u+f(u) \parallel +
\parallel  u_t
\parallel_{L^2(\Gamma)}\right) d\tau
\leq (E(u(t_0)) - E(\psi))^\theta. \label{3.53}\ee Since $E(u) -
E(\psi) \geq 0$,  we have \be \int_{t_0}^t \left(\parallel -\Delta
u+f(u) \parallel + \parallel u_t
\parallel_{L^2(\Gamma)}\right) d\tau < +\infty, \label{3.54}\ee which implies that
for all $t\geq t_0$, \be \int^t _{t_0}\|u_t\|d \tau \leq C.\ee This
easily yields that as $t\rightarrow +\infty$, $u(x, t)$ converges in
$L^2(\Omega)$. Since the orbit is compact in $H^1(\Omega)$, we can
deduce from uniqueness of limit that (\ref{1.5a}) holds, and the
theorem is proved.

 (3).  It
follows from (\ref{3.ee}) that for  any $ \varepsilon>0$ with
$\varepsilon < \beta_0$,  there exists  an integer $N$ such that
when $n \geq N$, \bea && \parallel
 u(\cdot,t_n) - \psi
 \parallel \leq   \|u(\cdot, t_n)-\psi\|_{H^1(\Omega)}< \frac{\varepsilon}{2}, \label{3.555}\\
 &&\frac{1}{\theta } (E(u(t_n)) - E(\psi))^\theta <
 \frac{\varepsilon}{4}. \label{3.56}\eea
Define
 \be \bar{t}_n = \sup \{\ t > t_n \mid \ \parallel
u(\cdot,s) - \psi \parallel_{H^1(\Omega)} < \beta _0,\ \forall s\in
[t_n , t] \} \label{3.57}\ee It follows from (\ref{3.555}) and
continuity of the orbit in $H^1(\Omega)$ for $t>0$ that $\bar{t}_n
>t_n$ for all
$n \geq N$. Then there are two possibilities:\\
(i). If there exists $n_0\geq N$ such that $\bar{t}_{n_0}=+\infty$,
then from the previous discussions in (1)(2), the
theorem is proved.\\
(ii) Otherwise, for all $n\geq N$, we have $t_n < \bar{t}_n <
+\infty$, and for all $t\in [t_n, \bar{t}_n]$, $E(\psi)<E(u(t))$.
Then from (\ref{3.53}) with $t_0$ being replaced by $t_n$, and $t$
being replaced by $\bar{t}_n$
 we deduce that
\be \int_{t_n}^{\bar{t}_n} \parallel\! u_t \!\parallel d\tau \leq
\frac{2 }{ \theta } (E(u(t_n)) - E(\psi))^\theta <
\frac{\varepsilon}{2}. \label{3.59}\ee Thus, it follows that \be
\parallel\! u(\bar{t}_n) - \psi \!
\parallel\ \leq\ \ \parallel\! u(t_n) - \psi
\!\parallel + \int_{t_n}^{\bar{t}_n}
\parallel\! u_t \!\parallel d\tau < \varepsilon
\label{3.60}\ee which implies that when $n\rightarrow +\infty $, \be
u(\bar{t}_n) \rightarrow \psi, \qquad \text{in} \ \ L^2(\Omega).\ee
Since $\bigcup_{t\geq\delta}u(t)$ is relatively compact in
$H^1(\Omega)$, there exists a subsequence of $\{u(\bar{t}_n)\}$,
still denoted  by $\{u(\bar{t}_n)\}$ converging to $\psi$ in
$H^1(\Omega)$. Namely, when $n$ is sufficiently large, \be\|
u(\bar{t}_n) - \psi \|_{H^1(\Omega)}\ < \beta_0,\ee which
contradicts the definition of $\bar{t}_n$ that $\|u(\cdot,
\bar{t}_n)-\psi\|_{H^1(\Omega)}=\beta_0$.

Thus, Theorem \ref{con2} is proved.

\noindent\textbf{Acknowledgement:} The author is indebted to Prof.
S. Zheng for his enthusiastic help with this paper. This work is
supported by the NSF of China under grant No.~10371022, by Chinese
Ministry of Education under the grant No. 20050246002, by the Key
Laboratory of Mathematics for Nonlinear Sciences sponsored by
Chinese Ministry of Education and the Graduate Student Innovation
Foundation of Fudan University.


\end{document}